\documentclass{amsart}
\usepackage[T1]{fontenc}
\usepackage{textcomp}
\usepackage[utf8x]{inputenc}
\usepackage[swedish,english]{babel}
\usepackage{ucs}
\usepackage{graphicx}
\usepackage{mathrsfs}
\usepackage{rotating} 
\usepackage{tikz}
\usepackage{tikz-cd}
\usepackage{amssymb}
\usepackage[titletoc,toc]{appendix}
\numberwithin{equation}{section} 

\DeclareFontFamily{U}{BOONDOX-calo}{\skewchar\font=45 }
\DeclareFontShape{U}{BOONDOX-calo}{m}{n}{
  <-> s*[1.05] BOONDOX-r-calo}{}
\DeclareFontShape{U}{BOONDOX-calo}{b}{n}{
  <-> s*[1.05] BOONDOX-b-calo}{}
\DeclareMathAlphabet{\mathcalboondox}{U}{BOONDOX-calo}{m}{n}
\SetMathAlphabet{\mathcalboondox}{bold}{U}{BOONDOX-calo}{b}{n}
\DeclareMathAlphabet{\mathbcalboondox}{U}{BOONDOX-calo}{b}{n}
\newtheorem{thm}{Theorem}[section]

\newtheorem{lem}[thm]{Lemma}
\newtheorem{prop}[thm]{Proposition}

\theoremstyle{definition}

\theoremstyle{remark}

\numberwithin{equation}{section}

\newcommand{\Pn}[1]{\mathbb{P}^{#1}}

\newcommand{\derham}[2]{H^{#1} (#2)}

\newcommand{\hyptot}{\mathcal{H}\!\mathcalboondox{y}\! \mathcalboondox{p}_{3,1}[2]}
\newcommand{\Qt}{\mathcal{Q}[2]}
\newcommand{\Qord}{\mathcal{Q}_{\mathrm{ord}}[2]}
\newcommand{\Qflx}{\mathcal{Q}_{\mathrm{flx}}[2]}
\newcommand{\Qbtg}{\mathcal{Q}_{\mathrm{btg}}[2]}
\newcommand{\Qhflx}{\mathcal{Q}_{\mathrm{hfl}}[2]}
\newcommand{\Qordcl}{\mathcal{Q}_{\overline{\mathrm{ord}}}[2]}
\newcommand{\Qbtgcl}{\mathcal{Q}_{\overline{\mathrm{btg}}}[2]}

\newcommand{\Htwo}[1]{\mathcal{H}\!\mathcalboondox{y}\! \mathcalboondox{p}_{#1}[2]}
\newcommand{\Honetwo}[1]{\mathcal{H}\!\mathcalboondox{y}\! \mathcalboondox{p}_{#1,1}[2]}
\newcommand{\Holtwo}[1]{\mathcal{H}\!\mathcalboondox{o}\! \mathcalboondox{l}_{#1}[2]}
\newcommand{\Hol}[1]{\mathcal{H}\!\mathcalboondox{o}\! \mathcalboondox{l}_{#1}}
\newcommand{\Holtwolambda}[2]{\mathcal{H}\!\mathcalboondox{o}\! \mathcalboondox{l}_{#1}^{#2}[2]}
\newcommand{\Hollambda}[2]{\mathcal{H}\!\mathcalboondox{o}\! \mathcalboondox{l}_{#1}^{#2}}
\newcommand{\Holtwolambdabar}[2]{\overline{\mathcal{H}\!\mathcalboondox{o}\! \mathcalboondox{l}}_{#1}^{#2}[2]}

\newcommand{\ftwo}{\mathbb{F}_2}
\newcommand{\symp}[1]{\mathrm{Sp}\! \left( #1 \right)}
\newcommand{\pic}{\mathrm{Pic}}
\newcommand{\jac}{\mathrm{Jac}}

\newcommand{\weyl}[1]{W(#1)}

\makeindex
\begin{document}

\title[{Cohomology of moduli spaces}]
{Equivariant cohomology of moduli spaces of genus three curves with level two structure}%
\author{Olof Bergvall, olofberg@math.su.se}%

\begin{abstract}
 We compute cohomology of the moduli space of genus three curves with level two structure
 and some related spaces.
 In particular, we determine the cohomology groups of the moduli space of plane quartics
 with level two structure as representations of the symplectic group on a six dimensional vector space
 over the field of two elements.
 We also make the analogous computations for some related spaces such as moduli spaces
 of genus three curves with a marked points and strata of the moduli space of
 Abelian differentials of genus three.
\end{abstract}

\maketitle
\tableofcontents


\section{Introduction}

The purpose of this note is to compute the de Rham cohomology cohomology 
(with rational coefficients) of various moduli spaces
of curves of genus $3$ with level $2$ structure. The group $\symp{6,\ftwo}$ acts
on the set of level $2$ structures of a curve. This action induces actions on
the various moduli spaces which in turn yields actions on the cohomology groups
which thus become $\symp{6,\ftwo}$-representations and our goal is to describe these representations.

The moduli spaces under consideration are essentially of three different types.
First of all, we have the moduli space $\mathcal{M}_3[2]$ of genus $3$ curves with
level $2$ structure and some natural loci therein. This will be our main object of
study.
Secondly we will consider the moduli space $\mathcal{M}_{3,1}[2]$ of genus $3$ curves
with level $2$ structure and one marked point and some of its subspaces.
Thirdly, we have the moduli space $\Holtwo{3}$ of genus $3$ curves with level
$2$ structure marked with a holomorphic (i.e. Abelian) differential and some related spaces,
e.g. the moduli space of genus $3$ curves marked with a canonical divisor.
There are many constructions, some classical and some new, relating the various spaces
and which will provide essential information for our cohomological computations.

The plan of the paper is as follows. 
In Section~\ref{backgroundsec} we give the basic
definitions and sketch some of the classical theory around genus $3$ curves and
their level $2$ structures.
In Section~\ref{markedsec} we sketch a construction, due to Looijenga \cite{looijenga},
which expresses some natural loci in $\mathcal{M}_{3,1}$ in terms of arrangements
of tori and hyperplanes and we use this description to compute the cohomology of these loci.
Hyperelliptic curves will be somewhat peripheral in this note but we give a discussion
in Section~\ref{hyperellipticsection}.
In Section~\ref{abeliansec} we recall some constructions and results regarding strata of
moduli spaces of Abelian differentials, essentially due to Kontsevich and Zorich \cite{kontsevichzorich},
and we make cohomological computations of these strata in genus $3$.
The core of the paper is Section~\ref{quarticsec} where we compute the cohomology
of the moduli space $\Qt$ of plane quartics with level $2$ structure as a representation
of $\symp{6,\ftwo}$. Finally, in Section~\ref{mthreesec} we make some comments around the
cohomology of $\mathcal{M}_3[2]$.

The main results of this note are presented in Section~\ref{tablesec} in the form of tables.
However, for convenience (and for readers not interested in the representation structure of the cohomology)
we also give some results in the form of Poincaré polynomials, for instance in Theorems~\ref{ordpoinc},
\ref{btgpoinc} and \ref{hyppoinc}.

\subsection*{Acknowledgements}
The author would like to thank Carel Faber and Jonas Bergström for helpful discussions
and comments and Orsola Tommasi for pointing out the papers \cite{fullartonputman} and \cite{looijengamondello}.
Some of the contents in this note is part of the PhD thesis \cite{bergvallthesis}, written
at Stockholms universitet, and parts of the research was carried out at Humboldt-Universität zu Berlin
and made possible by the Einstein foundation.

\section{Background}
\label{backgroundsec}

We work over over the field of complex numbers.

\subsection{Level structures}
Let $C$ be a smooth and irreducible curve of genus $g$ and let $\jac(C)$ denote its Jacobian.
For any positive integer $n$ there is an isomorphism
\begin{equation*}
 \jac (C)[n] \cong \left( \mathbb{Z}/n\mathbb{Z} \right)^{2g},
\end{equation*}
where $\jac (C)[n]$ denotes the subgroup of $n$-torsion elements in $\jac (C)$.
A \emph{symplectic level $n$ structure} on $C$ is an ordered basis $(D_1, \ldots, D_{2g})$
of $\jac (C)[n]$ such that the Weil pairing has matrix of the form
\begin{equation*}
 \left(
 \begin{array}{cc}
  0 & I_g \\
  -I_g & 0
 \end{array}
 \right)
\end{equation*}
with respect to this basis, where $I_g$ denotes the identity matrix of size $g \times g$.
We will often drop the adjective ``symplectic'' and simply say ``level $n$ structure''.
There is a moduli space of curves of genus $g$ with level $n$ structure which we
denote by $\mathcal{M}_g[n]$. For $n \geq 3$ it is fine but not for $n=2$ since
a level $2$ structure on a hyperelliptic curve is preserved by the hyperelliptic involution.
The symplectic group $\mathrm{Sp}(2g,\mathbb{Z}/n\mathbb{Z})$ acts 
on $\mathcal{M}_g[n]$ by changing the level $n$ structure.

\subsection{Curves of genus three}
Suppose that $C$ is of genus $3$. If $C$ is not hyperelliptic, then a choice 
of basis of its space of global sections gives an embedding of $C$ into
the projective plane as a curve of degree $4$. As is easily seen via
the genus-degree formula, every smooth plane quartic curve is of genus $3$
and we thus have a decomposition
\begin{equation*}
 \mathcal{M}_3[n] = \mathcal{Q}[n] \sqcup \mathcal{H}\!\mathcalboondox{y}\! \mathcalboondox{p}_3[n], 
\end{equation*}
where $\mathcal{Q}[n]$ denotes the \emph{quartic locus} and 
$\mathcal{H}\!\mathcalboondox{y}\! \mathcalboondox{p}_3[n]$ denotes the \emph{hyperelliptic locus}.

From now on we shall specialize to the case $n=2$. The locus $\mathcal{Q}[2]$
is by far the more complicated of the two loci and its investigation will therefore
take up most of this note, but we will also consider hyperelliptic curves in
Section~\ref{hyperellipticsection}.

There is a close relationship between level $2$ structures on a plane
quartic and its bitangents. More precisely, if $C \in \mathbb{P}^2$ is a plane
quartic curve and $B \in \mathbb{P}^2$ is a bitangent line of $C$, then
$C \cdot B = 2P + 2Q$ for some points $P$ and $Q$ on $C$. Thus,
if we set $D=P+Q$ then $2D = K_C$. Divisors $D$ with the property that $2D=K_C$
are called \emph{theta characteristics}. A theta characteristic $D$ is
called \emph{even} or \emph{odd} depending on whether $h^0(D)$ is even
or odd and it can be shown that there is a bijective correspondence
between the set of odd theta characteristics of $C$ and the set of
bitangents of $C$ given by the construction above. 

Given two theta characteristics
$D$ and $D'$ we obtain a $2$-torsion element by taking the difference $D-D'$.
This gives the set $\Theta$ of theta characteristics on $C$ the structure
of a $\jac (C)[2]$-torsor. The union $V=\jac (C)[2] \sqcup \Theta$ is
thus a vector space over $\ftwo = \mathbb{Z}/2\mathbb{Z}$ of dimension $7$.
Alternatively, we can describe $V$ as the $2$-torsion subgroup
of $\mathrm{Pic}(C)/\mathbb{Z}K_C$.

An ordered basis $\theta$ of $V$ consisting of odd theta characteristics is called
an \emph{ordered Aronhold basis} if it has the property that the expression
\begin{equation*}
 h(D) \quad \mathrm{mod} \, 2
\end{equation*}
only depends on the number of elements in $\theta$ required to express $D$
for any theta characteristic $D$.

\begin{prop}
 There is a bijection between the set of ordered Aronhold bases on $C$
 and the set of level $2$ structures on $C$.
\end{prop}

For a proof, see \cite{dolgachevortland} or \cite{grossharris}.

Thus, we may think of a level $2$ structure on $C$ as an ordered Aronhold basis
of odd theta characteristics on $C$. Since odd theta characteristics are
cut out by bitangents we can also think about level $2$ structures in terms
of ordered sets of seven bitangents (but we must then bear in mind that not
every ordered set of seven bitangents corresponds to a level $2$ structure).

\subsection{Point configurations in the projective plane}
Let $P_1, \ldots, P_7$ be seven points in $\mathbb{P}^2$. We say that
the points are in \emph{general position} if there is no ``unexpected'' curve
passing through any subset of them, i.e. if
\begin{itemize}
 \item no three of the points lie on a line and
 \item no six of the points lie on a conic.
\end{itemize}
We denote the moduli space of ordered septuples of points in general position in 
$\mathbb{P}^2$ up to projective equivalence by $\mathcal{P}^2_7$.

Given seven points in general position in $\mathbb{P}^2$ there is a
net $\mathcal{N}$ of cubics passing through the points. The set of
singular members of $\mathcal{N}$ is a plane curve $T$
of degree $12$ with $24$ cusps and $28$ nodes. The dual $T^{\vee} \subset \mathcal{N}^{\vee} \cong \mathbb{P}^2$ 
is a smooth plane quartic curve. 
Another way to obtain a genus $3$ is by taking the set $S$ of singular points
of members of $\mathcal{N}$. The set $S$ is a sextic curve
with ordinary double points precisely at $P_1, \ldots, P_7$.
From this information it is easy to see, via the genus-degree formula,
that $S$ has geometric genus $3$. One can also show that the map $\sigma$
sending a point $P$ in $S$ to the unique member of $\mathcal{N}$ with
a singularity at $P$ is a birational isomorphism from $S$ to $T$.

\subsection{Del Pezzo surfaces of degree two}
 Recall that a \emph{Del Pezzo surface} is a smooth and projective algebraic variety of dimension
 two such that its anticanonical class is ample. The \emph{degree} of a Del Pezzo surface $S$ is the
 self intersection number of its canonical class, $K_S^2$.
 
 Given seven points $P=(P_1, \ldots, P_7)$ in general position in $\mathbb{P}^2$,
 the blow-up $X=\mathrm{Bl}_P\mathbb{P}^2$ is a Del Pezzo surface of degree $2$.
 Moreover, every Del Pezzo surface of degree $2$ can be realized as such a blow-up,
 see \cite{manin}. 
 We denote the blow-up map by $\pi: X \to \mathbb{P}^2$.
 However, the points $P_1, \ldots, P_7$ do only give us the
 Del Pezzo surface $X$ - we also get the seven exceptional curves $E_1, \ldots, E_7$.
 Together with the strict transform $L$ of a line in $\mathbb{P}^2$ they determine
 a basis for the Picard group of $X$
 \begin{equation*}
  \mathrm{Pic}(X) = \mathbb{Z}L \oplus \mathbb{Z}E_1 \oplus \cdots \oplus \mathbb{Z}E_7.
 \end{equation*}
 The intersection theory is given by
 \begin{equation*}
  L^2=1, \quad E_i^2=-1, \quad L \cdot E_i = E_i \cdot E_j =0, \quad i\neq j.
 \end{equation*}
 Not every ordered basis of $\mathrm{Pic}(X)$ comes from a blow-up as above.
 Bases which do arise in this way are called \emph{geometric markings}. 
 Two geometrically marked Del Pezzo surfaces $(X,E_1, \ldots, E_7)$ and
 $(X',E'_1, \ldots, E'_7)$ are isomorphic if there is an isomorphism of surfaces
 $\phi:X \to X'$ such that $\phi^*(E'_i)=E_i$ for all $i$.
 We denote the moduli
 space of geometrically marked Del Pezzo surfaces of degree $2$ by
 $\mathcal{DP}_2^{\mathrm{gm}}$.

 Given a quartic $C \subset \mathbb{P}^2$ we can obtain a Del Pezzo surface $X$
 of degree $2$ as the double cover of $\mathbb{P}^2$ ramified along $C$. 
 Moreover, every Del Pezzo surface of degree $2$ can be realized as such a
 double cover, see \cite{kollar}.
 We let $p:X \to \mathbb{P}^2$ denote the covering map and let $\iota$ denote
 the involution exchanging the two sheets.
 If $E_1, \ldots, E_7$ is a geometric marking of $X$ then
 $p(E_1), \ldots, p(E_7)$ is an ordered Aronhold set of bitangents of $C$.
 
 We have thus seen how to obtain a geometrically marked Del Pezzo surface of degree
 $2$ both from seven ordered  points in general position and from a plane quartic
 curve with level $2$ structure and we have also seen how to
 obtain the quartic curve directly from the seven points. We summarize the situation
 in the diagram below.
\[
\begin{tikzcd}[column sep=1.5em]
&  X \arrow[swap]{dl}{\pi} \arrow{dr}{p} & \\
 \mathbb{P}^2_{\mathrm{pts}} \arrow[dashed]{rr}{\sigma^{\vee}} & & \mathbb{P}^2_{\mathrm{curve}} 
\end{tikzcd}
\]
Here $\sigma^{\vee}$ denotes the composition of $\sigma$ and the dualization map.
In each of the spaces we have a copy of the curve $C$: in $\mathbb{P}^2_{\mathrm{curve}}$ 
we have the actual curve $C$, in $X$ we obtain an isomorphic copy of $C$ by taking the
fixed locus of the involution $\iota$ and in $\mathbb{P}^2_{\mathrm{pts}}$ we have
a sextic model $S$ of $C$ with seven double points.

\begin{thm}[van Geemen, \cite{dolgachevortland}]
 The above construction yields $\mathrm{Sp}(6,\mathbb{F}_2)$-equivariant isomorphisms
 \[
\begin{tikzcd}[column sep=1.5em]
&  \, \, \, \mathcal{DP}^{\mathrm{gm}}_2 \arrow[leftrightarrow]{dl} \arrow[leftrightarrow]{dr} & \\
 \mathcal{P}^2_7 \arrow[leftrightarrow]{rr} & & \mathcal{Q}[2] 
\end{tikzcd}
\]
\end{thm}

\section{Curves and surfaces with marked points}
\label{markedsec}

 \subsection{Genus three curves with marked points}
 We now turn our attention to the moduli space $\mathcal{M}_{3,1}[2]$ 
 of genus $3$ curves with level $2$ structure and one marked point.
 Also in this case we have a decomposition
 \begin{equation*}
  \mathcal{M}_{3,1}[2] = \mathcal{Q}_{1}[2] \sqcup \mathcal{H}\!\mathcalboondox{y}\! \mathcalboondox{p}_{3,1}[2]
 \end{equation*}
 into a quartic locus and a hyperelliptic locus. However, in this case
 there is also a natural decomposition of the quartic locus in terms
 of the behaviour of the tangent line at the marked point.
 
 Let $C$ be a plane quartic curve, let $P$ be a point on $C$ and let $T_P \subset \mathbb{P}^{2}$ denote
the tangent line of $C$ at $P$. Since $C$ is of degree $4$, B\'ezout's theorem tells us
that the intersection product $C \cdot T_P$ will consist of $4$ points. There are four possibilities:
\begin{itemize}
 \item[(\emph{i})] $T_P \cdot C = 2P + Q + R$ where $Q$ and $R$ are two distinct points on $C$, both different from $P$.
 In this case, $T_P$ is called an \emph{ordinary tangent line} of $C$ and $P$ is called an \emph{ordinary point} of $C$.
 \item[(\emph{ii})] $T_P \cdot C = 2P + 2Q$ where $Q \neq P$ is a point on $C$. In this case,
 $T_P$ is called a \emph{bitangent} of $C$ and $P$ is called a \emph{bitangent point} of $C$.
 \item[(\emph{iii})] $T_P \cdot C = 3P +Q$ where $Q \neq P$ is a point on $C$. In this case,
 $T_P$ is called a \emph{flex line} of $C$ and $P$ is called a \emph{flex point} of $C$.
 \item[(\emph{iv})]  $T_P \cdot C = 4P$. In this case, $T_P$ is called a \emph{hyperflex line} of
 $C$ and $P$ is called a \emph{hyperflex point} of $C$.
\end{itemize}
This yields a decomposition of $\mathcal{Q}_1[2]$
\begin{equation*}
 \mathcal{Q}_1[2] = 
 \mathcal{Q}_{\mathrm{ord}}[2] \sqcup 
 \mathcal{Q}_{\mathrm{flx}}[2] \sqcup
 \mathcal{Q}_{\mathrm{btg}}[2] \sqcup
 \mathcal{Q}_{\mathrm{hfl}}[2]
\end{equation*}
into a locus of curves marked with an ordinary, flex, bitangent and hyperflex
point, respectively.

\subsection{Del Pezzo surfaces of degree two with marked points}
Let $X$ be a Del Pezzo surface of degree $2$.
Recall that we can realize $S$ both as a double cover $p:S \to \mathbb{P}^2$ ramified over a
plane quartic $C$ and as the blowup $\pi:X \to \mathbb{P}^2$ in seven points $P_1, \ldots, P_7$
in general position. Also recall that there is an involution $\iota$
of $X$ and that we can identify the fixed points of $\iota$ in $X$ with $p^{-1}(C)$. 
We shall now give another characterization of the fixed points of $\iota$.

A curve $A \subset X$ in the anticanonical linear system $|-K_X|$ is called
an \emph{anticanonical curve}. The anticanonical class $-K_X=3L-E_1-\cdots-E_7$ corresponds to 
the linear system $\mathcal{C}$ of cubics in $\mathbb{P}^2{2}$ passing
through $P_1, \ldots, P_7$. The curve $B=\pi(p^{-1}(C))$ consists of all 
the singular points of members of $\mathcal{C}$.
We thus see that a point $Q \in X$ is a point of $p^{-1}(C)$ if and only if
there is a unique anticanonical curve $A$ with a singularity at $Q$.
Note that since $A$ is isomorphic to a singular plane cubic, its
irreducible components will be rational.

By the above construction we have that if $(C,P)$ is a plane quartic with a marked point, 
the double cover $p:X \to \mathbb{P}^2$ ramified along $C$ naturally becomes equipped with an
anticanonical curve $A$ with a singularity at the inverse image of $P$. Thus, if we introduce the moduli
space $\mathcal{DP}_{2,\mathrm{a}}^{\mathrm{gm}}$ of geometrically marked Del Pezzo surfaces of degree $2$ 
marked with a singular point of an anticanonical curve
we have an isomorphism $\mathcal{Q}_1[2] \cong \mathcal{DP}_{2,\mathrm{a}}^{\mathrm{gm}}$. 

Since $p^{-1}(C) \sim -2K_X$ it follows that
\begin{equation*}
 A.p^{-1}(C)=(-K_X).(-2K_X)=4.
\end{equation*}
We have that $A$ intersects $p^{-1}(C)$
with multiplicity at least $2$ so $p(A)$ is a tangent to $C$.
The anticanonical curve $A$ can be of the following types.
\begin{itemize}
 \item[(\emph{i})] The anticanonical curve $A$ can be an irreducible curve with a node. Then $p(A)$ intersects $C$
 with multiplicity $2$ at $P$ so $p(A)$ is either an ordinary tangent line or
 a bitangent. But we have shown that the inverse image of a bitangent under $p$ consists of
 two exceptional curves which are conjugate under $\iota$ and we conclude that $p(A)$ is an ordinary
 tangent line. We may thus identify the locus $\mathcal{DP}_{2,\mathrm{n}}^{\mathrm{gm}} \subset \mathcal{DP}_{2,\mathrm{a}}^{\mathrm{gm}}$ consisting of surfaces
 such that the anticanonical curve through the marked point is irreducible and nodal with the locus $\mathcal{Q}_{\mathrm{ord}[2]} \subset \mathcal{Q}_1[2]$
 consisting of curves whose marked point is ordinary.
 \item[(\emph{ii})] The anticanonical curve $A$ can be an irreducible curve with a cusp. Then $p(A)$ intersects $C$
 with multiplicity $3$ at $P$ so $p(A)$ must be a flex line. 
 We may thus identify the locus $\mathcal{DP}_{2,\mathrm{c}}^{\mathrm{gm}} \subset \mathcal{DP}_{2,\mathrm{a}}^{\mathrm{gm}}$ consisting of surfaces
 such that the anticanonical curve through the marked point is irreducible and cuspidal with the locus $\mathcal{Q}_{\mathrm{flx}}[2] \subset \mathcal{Q}_1[2]$
 consisting of curves whose marked point is a genuine flex point.
 \item[(\emph{iii})] The anticanonical curve $A$ can consist of two rational curves intersecting with multiplicity one at $P$.
 Thus, the cubic $\pi(A)$ must be the product of a conic through five of the points $P_1, \ldots,P_7$ with a line through the 
 remaining two. Hence, $A$ consists of two conjugate exceptional curves and $p(A)$ is a genuine bitangent.
 We may thus identify the locus $\mathcal{DP}_{2,\mathrm{t}}^{\mathrm{gm}} \subset \mathcal{DP}_{2,\mathrm{a}}^{\mathrm{gm}}$ consisting of surfaces
 such that the anticanonical curve through the marked point consists of two rational curves intersecting transversally 
 in two distinct points with the locus $\mathcal{Q}_{\mathrm{btg}}[2] \subset \mathcal{Q}_1[2]$
 consisting of curves whose marked point is a genuine bitangent point.
 \item[(\emph{iv})] The anticanonical curve $A$ can consist of two rational curves intersecting with multiplicity two at $P$.
 An analysis similar to the one above shows that $p(A)$ is then a hyperflex line.
 We may thus identify the locus $\mathcal{DP}_{2,\mathrm{d}}^{\mathrm{gm}} \subset \mathcal{DP}_{2,\mathrm{a}}^{\mathrm{gm}}$ consisting of surfaces
 such that the anticanonical curve through the marked point consists of two rational curves with a double intersection
 with the locus $\mathcal{Q}_{\mathrm{hfl}}[2] \subset \mathcal{Q}_1[2]$
 consisting of curves whose marked point is a hyperflex point.
\end{itemize}
In \cite{looijenga}, Looijenga gave descriptions of each of these loci in terms of
arrangements. In order to give his results, we need to investigate the
Picard group of $X$ in a bit more detail.

The Del Pezzo surface $X$
 has exactly $56$ exceptional curves which can be described as follows.
 \begin{itemize}
  \item[(\emph{i})] The $7$ exceptional curves $E_i$.
  \item[(\emph{ii})] The $21$ strict transforms of lines between two points $P_i$ and $P_j$. 
  The classes of these curves are given by $L-E_i-E_j$.
  \item[(\emph{iii})] The $21$ strict transforms of conics through all but two points $P_i$ and $P_j$.
  The classes of these curves are given by $2L-E_1-\cdots-E_7+E_i+E_j$.
  \item[(\emph{iv})] The $7$ cubics through $P_1, \ldots, P_7$ with a singularity in one of the points $P_i$.
  The classes of these curves are given by $3L-E_1-\cdots-E_7-E_i$.
 \end{itemize}
 We denote the set of these classes by $\mathscr{E}$.

 The involution $\iota$ fixes the anticanonical class $K_X$.
 We denote the orthogonal complement of $K_X$ in $\mathrm{Pic}(X)$ by
 $K_X^{\perp}$. The involution $\iota$ acts as $-1$ on $K_X^{\perp}$. 
 The elements $\alpha$ in $K_X^{\perp}$ such that
 $\alpha^2=-2$ form a root system $\Phi$ of type $E_7$. We denote
 the Weyl group of $E_7$ by $\weyl{E_7}$. The group $\weyl{E_7}$ is
 isomorphic to $\symp{6, \mathbb{F}_2} \times \mathbb{Z}_2$ where $\mathbb{Z}_2$
 is the group of two elements generated by $\iota$. We denote the
 quotient $\weyl{E_7}/\langle \iota \rangle$ by $\weyl{E_7}^+$.
 
 \subsubsection{The irreducible nodal case}
 Let $X$ be a geometrically marked Del Pezzo surface of degree $2$ and let
 $P$ be a point of $X$ such that there is a unique rational anticanonical curve $A$ on $X$
 which is nodal at $P$. 
 The Jacobian $\mathrm{Jac}(A)$ is isomorphic to $k^*$ as a group, see Chapter II.6 of 
 \cite{hartshorne}, and the restriction homomorphism
 \begin{equation*}
  \mathrm{Pic}(X) \to \mathrm{Pic}(A),
 \end{equation*}
 induces a homomorphism
 \begin{equation*}
  r: K_X^{\perp} \to \mathrm{Jac}(A).
 \end{equation*}
 Recall that $K_S^{\perp}$ is a lattice isometric to the $E_7$-lattice $L_{E_7}$. We thus see
 that $r$ is an element of the $7$-dimensional algebraic torus 
 $T=\mathrm{Hom}(K_S^{\perp},\mathrm{Jac}(A)) \cong (k^*)^7$ and we have a natural action
 of $\weyl{E_7}$ on $T$ via its action on $K_S^{\perp}$. 

Every root $\alpha$ in $\Phi$ determines a multiplicative character on $T$ by evaluation,
i.e. by sending an element $\chi \in T$ to $\chi(\alpha) \in k^*$.
Let
\begin{equation*}
 T_{\alpha} = \{\chi \in T | \chi(\alpha)=1\}
\end{equation*}
and define
\begin{equation*}
 D_{E_7} = \bigcup_{\alpha \in \Phi(E_7)} T_{\alpha},
\end{equation*}
and let $T_{E_7}$ be the complement $T \setminus D_{E_7}$. 
We remark that $D_{E_7}$ is the toric arrangement associated to
the root system $E_7$.

\begin{thm}[Looijenga \cite{looijenga}]
\label{TE7prop}
 There is a $\weyl{E_7}^+$-equiv\-ariant isomorphism
 \begin{equation*}
  \mathcal{DP}^{\mathrm{gm}}_{2,\mathrm{n}} \to \{\pm 1 \} \setminus T_{E_7}.
 \end{equation*}
\end{thm}
Since $\mathcal{Q}_{\mathrm{ord}}[2]$ is isomorphic to $\mathcal{DP}^{\mathrm{gm}}_{2,\mathrm{n}}$
and $\weyl{E_7}^+$ is isomorphic to $\symp{6, \mathbb{F}_2}$, it follows
that there is a $\symp{6, \mathbb{F}_2}$-equivariant isomorphism
$\mathcal{Q}_{\mathrm{ord}}[2] \cong \{\pm 1 \} \setminus T_{E_7}$.

\subsubsection{The other cases}
The three other cases have similar descriptions. For instance,
if we let $V_{E_7}$ denote the complement of the hyperplane arrangement
associated to $E_7$ we have the following.

\begin{thm}[Looijenga \cite{looijenga}]
\label{VE7prop}
 There is a $\weyl{E_7}^+$-equiv\-ariant isomorphism
 \begin{equation*}
  \mathcal{DP}^{\mathrm{gm}}_{2,\mathrm{c}} \to \mathbb{P}\left( V_{E_7} \right).
 \end{equation*}
\end{thm}
It follows that there is a $\symp{6, \mathbb{F}_2}$-equivariant isomorphism
$\mathcal{Q}_{\mathrm{flx}}[2] \cong \mathbb{P}(V_{E_7})$.

 In order to state the results for the remaining two cases we need to introduce a little bit of notation.
 Let $E$ be an exceptional curve. Then $E+\iota(E) = K_X$. We denote
 the orthogonal complement of $\langle E, \iota(E) \rangle$ in $\pic(X)$
 by $\langle E, \iota(E) \rangle^{\perp}$.
 Since $K_X \in \langle E, \iota(E) \rangle$ we have
 $\langle E, \iota(E) \rangle^{\perp} \subset K_X^{\perp}$ and 
 $\Phi_E=\Phi \cap \pic^{\perp}_{\langle E, \iota(E) \rangle}(S)$ is a subrootsystem of type $E_6$.
 We denote the Weyl group of $E_6$ by $\weyl{E_6}$. We denote the complement
 of the toric arrangement associated to $E_6$ by $T_{E_6}$ and we denote the complement
 of the hyperplane arrangement associated to $E_6$ by $V_{E_6}$. The elements of $\mathscr{E}$
 are in bijective correspondence with the cosets in the quotient $W(E_7)/W(E_6)$ and for
 each $e \in \mathscr{E}$ we let $T_{E_6}(e)$ be an isomorphic copy of $T_{E_6}$.
 similarly, we let $\mathbb{P}(V_{E_6})(e)$ be an isomorphic copy of $\mathbb{P}(V_{E_6})$.
 We then have the following two results.
 
 \begin{thm}[Looijenga \cite{looijenga}]
\label{TE6prop}
 There are $\weyl{E_7}^+$-equiv\-ariant isomorphisms
 \begin{equation*}
  \mathcal{DP}^{\mathrm{gm}}_{2,\mathrm{t}} \to \{\pm1 \} \setminus \coprod_{e \in \mathscr{E}} T_{E_6}(e)
 \end{equation*}
 and
 \begin{equation*}
  \mathcal{DP}^{\mathrm{gm}}_{2,\mathrm{d}} \to \{\pm1 \} \setminus \coprod_{e \in \mathscr{E}} \mathbb{P}(V_{E_6})(e).
 \end{equation*}
\end{thm}
It follows that there are $\symp{6, \mathbb{F}_2}$-equivariant isomorphisms
$\mathcal{Q}_{\mathrm{btg}}[2] \cong \{\pm1 \} \setminus \coprod_{e \in \mathscr{E}} T_{E_6}(e)$
and $\mathcal{Q}_{\mathrm{hfl}}[2] \cong \{\pm1 \} \setminus \coprod_{e \in \mathscr{E}} \mathbb{P}(V_{E_6})(e)$.

\subsection{Cohomological computations}
We have thus seen how each of the four strata of $\mathcal{Q}_1[2]$ either can be described in
terms of  complements of toric arrangements or in terms of complements hyperplane arrangements.
In the affine hyperplane case, the necessary computations were carried out by Fleischmann and Janiszczak,
see \cite{fleischmannjaniszczak}. They present their results in terms of equivariant Poincaré polynomials
and one goes from the affine to the projective case by dividing their results by $1+t$.
In the toric case, the necessary computations were carried out by the author in
\cite{bergvalltor}. 

Since $W(E_7)=\symp{6, \mathbb{F}_2} \times \{\pm 1\}$ we have that each
representation of $W(E_7)$ either is a representation of $\symp{6, \mathbb{F}_2}$ times
the trivial representation of $\{\pm 1\}$ or a representation of $\symp{6, \mathbb{F}_2}$
times the alternating representation of $\{\pm 1\}$.
Thus, to go from the cohomology of the complement of an arrangement associated to $E_7$ one
simply takes the $\{\pm 1\}$-invariant part. This explains how we obtained the cohomology
of $\mathcal{Q}_{\mathrm{ord}}[2]$ and $\mathcal{Q}_{\mathrm{flx}}[2]$ given
in Table~\ref{Qordtable} and Table~\ref{Qflxtable}, respectively.
If one is only interested in the dimensions of the various cohomology
groups, these are more conveniently given as Poincaré polynomials.

\pagebreak[2]
\begin{thm}
\label{ordpoinc}
 The cohomology groups $H^i(\mathcal{Q}_{\mathrm{ord}}[2])$ and $H^i(\mathcal{Q}_{\mathrm{flx}}[2])$
 are both pure of Tate type $(i,i)$. The Poincaré polynomial of $\mathcal{Q}_{\mathrm{ord}}[2]$ is
 \begin{equation*}
  \begin{array}{rcl}
   P(\mathcal{Q}_{\mathrm{ord}}[2],t) & =  & 1 + 63t + 1638t^2 + 22680t^3 + 180089t^4 +\\
    & & + 820827t^5 + 2004512t^6 + 2064430t^7
  \end{array}
 \end{equation*}
 and the Poincaré polynomial of $\mathcal{Q}_{\mathrm{flx}}[2]$ is
 \begin{equation*}
 \begin{array}{rcl}
  P(\mathcal{Q}_{\mathrm{flx}}[2],t) & = & 1 + 62t + 1555t^2+ 20180t^3 + 142739t^4 + \\
  & & + 521198t^5 +  765765t^6.
 \end{array}
 \end{equation*}
\end{thm}

To obtain the cohomology of $\mathcal{Q}_{\mathrm{btg}}[2]$ from the cohomology
of $T_{E_6}$ we first have to induce from $W(E_6)$ and then take $\{\pm 1\}$-invariants.
Thus
\begin{equation*}
 H^i(\mathcal{Q}_{\mathrm{btg}}[2]) = \mathrm{Ind}_{W(E_6)}^{W(E_7)} \left( H^i(T_{E_7}) \right)^{\{\pm 1\}}.
\end{equation*}
The results are given in Table~\ref{Qbtgtable}. In an entirely analogous way one obtains 
the cohomology of $\mathcal{Q}_{\mathrm{hfl}}[2]$ from the cohomology of $\mathbb{P}(V_{E_7})$.
The results are given in Table~\ref{Qhflxtable}.

\pagebreak[2]
\begin{thm}
\label{btgpoinc}
 The cohomology groups $H^i(\mathcal{Q}_{\mathrm{btg}}[2])$ and $H^i(\mathcal{Q}_{\mathrm{hfl}}[2])$
 are both pure of Tate type $(i,i)$. The Poincaré polynomial of $\mathcal{Q}_{\mathrm{btg}}[2]$ is
 \begin{equation*}
  \begin{array}{rcl}
   P(\mathcal{Q}_{\mathrm{btg}}[2],t) & =  & 28 + 1176t + 19740t^2 + 168560t^3 + 768852t^4 +\\
   & & + 1774584t^5 + 1639540t^6
  \end{array}
 \end{equation*}
 and the Poincaré polynomial of $\mathcal{Q}_{\mathrm{hfl}}[2]$ is
 \begin{equation*}
 \begin{array}{rcl}
  P(\mathcal{Q}_{\mathrm{hfl}}[2],t) & = & 28 + 980t + 13300t^2 + 87500t^3 + 278992t^4 + \\
  & & + 344960t^5.
 \end{array}
 \end{equation*}
\end{thm}

Let $\mathcal{Q}_{\overline{\mathrm{ord}}}[2]$ be the union of $\mathcal{Q}_{\mathrm{ord}}[2]$
and $\mathcal{Q}_{\mathrm{flx}}[2]$ inside $\mathcal{Q}_1[2]$. By Looijenga's results \cite{looijenga}
we have that there is a $\symp{6, \mathbb{F}_2}$-equivariant short exact sequence
of mixed Hodge structures
\begin{equation*}
 0 \to H^i(\mathcal{Q}_{\overline{\mathrm{ord}}}[2]) \to 
 H^i(\mathcal{Q}_{\mathrm{ord}}[2]) \to
 H^{i-1}(\mathcal{Q}_{\mathrm{flx}}[2])(-1) \to 0.
\end{equation*}
Thus, $H^i(\mathcal{Q}_{\overline{\mathrm{ord}}}[2])$ is pure of Tate type $(i,i)$ and we
may easily deduce the structure as a $\symp{6, \mathbb{F}_2}$-representation
from Tables~\ref{Qordtable} and \ref{Qflxtable}. The result is given in Table~\ref{Qordcltable}.

\pagebreak[2]
\begin{thm}
 The cohomology group $H^i(\mathcal{Q}_{\overline{\mathrm{ord}}}[2])$ 
 is pure of Tate type $(i,i)$ and the Poincaré polynomial of $\mathcal{Q}_{\overline{\mathrm{ord}}}[2]$ is
 \begin{equation*}
  \begin{array}{rcl}
   P(\mathcal{Q}_{\mathrm{ord}}[2],t) & =  & 1 + 62t + 1576t^2 + 21125t^3 + 159909t^4 +\\
    & & + 678068t^5 + 1483314t^6 + 1302665t^7.
  \end{array}
 \end{equation*}
\end{thm}

Similarly, let $\mathcal{Q}_{\overline{\mathrm{btg}}}[2]$ be the union of $\mathcal{Q}_{\mathrm{btg}}[2]$
and $\mathcal{Q}_{\mathrm{hfl}}[2]$ inside $\mathcal{Q}_1[2]$. Again, by results of Looijenga \cite{looijenga}
we have that there is a $\symp{6, \mathbb{F}_2}$-equivariant short exact sequence
of mixed Hodge structures
\begin{equation*}
 0 \to H^i(\mathcal{Q}_{\overline{\mathrm{btg}}}[2]) \to 
 H^i(\mathcal{Q}_{\mathrm{btg}}[2]) \to
 H^{i-1}(\mathcal{Q}_{\mathrm{hfl}}[2])(-1) \to 0.
\end{equation*}
Thus, $H^i(\mathcal{Q}_{\overline{\mathrm{btg}}}[2])$ is pure of Tate type $(i,i)$ and we
may easily deduce the structure as a $\symp{6, \mathbb{F}_2}$-representation
from Tables~\ref{Qbtgtable} and \ref{Qhflxtable}. The result is given in Table~\ref{Qbtgcltable}.

\pagebreak[2]
\begin{thm}
 The cohomology group $H^i(\mathcal{Q}_{\overline{\mathrm{btg}}}[2])$ 
 is pure of Tate type $(i,i)$ and the Poincaré polynomial of $\mathcal{Q}_{\overline{\mathrm{btg}}}[2]$ is
 \begin{equation*}
  \begin{array}{rcl}
   P(\mathcal{Q}_{\mathrm{btg}}[2],t) & =  & 28 + 1148t + 18760t^2 + 155260t^3 + 681352t^4 +\\
    & & + 1495592t^5 + 1294580t^6.
  \end{array}
 \end{equation*}
\end{thm}

\section{Hyperelliptic curves}
\label{hyperellipticsection}
We shall now briefly turn out attention to the moduli spaces of hyperelliptic curves, 
 $\Htwo{3}$ and $\Honetwo{3}$.

 Let $C$ be a hyperelliptic curve of genus $g \geq 2$. Then $C$ can be realized
 as a double cover of $\Pn{1}$ ramified over $2g+2$ points. Moreover, if we pick
 $2g+2$ ordered points on $\Pn{1}$, the curve $C$ obtained as the double cover ramified over precisely
 those points is a hyperelliptic curve and the points also determine a level $2$-structure on $C$.
 However, not all level $2$-structures on $C$ arise in this way.
 Nevertheless, there is an intimate
 relationship between the moduli space $\Htwo{g}$ and the moduli space $\mathcal{M}_{0,2g+2}$ of
 $2g+2$ ordered points on $\Pn{1}$.

 \begin{thm}[Dolgachev and Ortland \cite{dolgachevortland}, Tsuyumine \cite{tsuyumine}]
  Let $\mathfrak{S}$ denote the set of left cosets $\symp{2g,\mathbb{F}_2}/S_{2g+2}$, let
  $[\sigma] \in \mathfrak{S}$ and let $X_{\sigma}=\mathcal{M}_{0,2g+2}$.
  Then
  \begin{equation*}
   \Htwo{g} \cong \coprod_{[\sigma] \in \mathfrak{S}} X_{[\sigma]},
  \end{equation*}
  and $\symp{2g,\mathbb{F}_2}$ acts transitively on the set of connected components of $\Htwo{g}$.
 \end{thm}

 Thus, the cohomology of $\Htwo{g}$ can be obtained by computing the cohomology of $\mathcal{M}_{0,2g+2}$
 as a $S_{2g+2}$-representation and then inducing up to $\symp{2g,\mathbb{F}_2}$.
 More precisely, we have
 \begin{equation*}
   H^i(\Htwo{g}) = \mathrm{Ind}_{S_{2g+2}}^{\symp{2g,\mathbb{F}_2}} \left( H^i(\mathcal{M}_{0,2g+2}) \right)
 \end{equation*}
 where we consider $H^i(\mathcal{M}_{0,2g+2}) $ as a $S_{2g+2}$-representation and $H^i(\Htwo{g})$ as a
 $\symp{2g,\mathbb{F}_2}$-representation.
 One way to compute the cohomology of $\mathcal{M}_{0,2g+2}$ is to
 make $S_{2g+2}$-equivariant point counts of $\mathcal{M}_{0,2g+2}$. Since $\mathcal{M}_{0,2g+2}$
 is isomorphic to a hyperplane arrangement, this will determine its cohomology, see \cite{dimcalehrer}. 
 For $\Htwo{3}$, these computations were carried out in \cite{bergvallpts} and the results
 are given, for convenience, in Table~\ref{Hyptable}.
 We also mention that $H^k(\Htwo{3})$ is pure of Tate type
 $(k,k)$. 
 
 \begin{thm}
 \label{hyppoinc}
  The Poincaré polynomial of $\Htwo{3}$ is
  \begin{equation*}
   P(\Htwo{3},t) = 36+720t +5580t^2 +20880t^3 +37584t^4 +25920t^5
  \end{equation*}
 \end{thm}

 The moduli space $\Honetwo{3}$ (as a coarse moduli space) is a $\Pn{1}$-fibration over $\Htwo{3}$ via the forgetful morphism.
 The Leray-Serre spectral sequence of this fibration degenerates at the second page and
 allows us to compute the cohomology of $\Honetwo{3}$, together with its mixed Hodge structure, as
 \begin{equation*}
  \derham{k}{\Honetwo{3}} = \derham{k-2}{\Htwo{3}}(\text{-}1) \oplus \derham{k}{\Htwo{3}}.
 \end{equation*}
 Thus, the cohomology of $\Honetwo{3}$ is easily obtained via Table~\ref{Hyptable}.
 
 \section{Moduli of Abelian differentials}
 \label{abeliansec}
Let $\Hol{g}$ denote the moduli space of pairs $(C,\omega)$ where
$C$ is a curve of genus $g$ and $\omega$ is a nonzero holomorphic $1$-form (i.e. an Abelian differential)
and  let $\Holtwo{g}$ denote the corresponding moduli space where the curves are also
marked with a level $2$ structure.
Kontsevich and Zorich \cite{kontsevichzorich} gave stratification of $\Holtwo{g}$ 
according to the multiplicities of the zeros of $\omega$ and we follow them
in order to obtain a corresponding stratification of $\Holtwo{g}$. More precisely,
let $\lambda=[\lambda_1, \ldots ,\lambda_n]$ be a partition of $2g-2$. 
Then there is a subspace $\Holtwolambda{g}{\lambda}$ consisting of equivalence classes
such that $\omega$ has exactly $n$ zeros with multiplicities prescribed by $\lambda$.
We now have
\begin{equation*}
 \Holtwo{g} = \coprod_{\lambda \vdash 2g-2} \Holtwolambda{g}{\lambda}.
\end{equation*}
Let $P(2g-2)$ denote the set of partitions of $2g-2$. The elements of $P(2g-2)$ are
partially ordered by refinement and under this order the partition $[2g-2]$ is the maximal
element. Let $\Holtwolambdabar{g}{\lambda}$ denote the closure of $\Holtwolambda{g}{\lambda}$
in $\Holtwo{g}$. Then
\begin{equation*}
 \Holtwolambdabar{g}{\lambda} = \coprod_{\lambda' \in [\lambda,[2g-2]]} \Holtwolambda{g}{\lambda'}.
\end{equation*}

The strata $\Hollambda{g}{\lambda}$ are not connected in general and
Kontsevich and Zorich \cite{kontsevichzorich} have given a complete description
of their connected components. In genus $3$, the result is exceptionally
simple (since there are no effective even theta characteristics in genus $3$).
More precisely, the strata $\Hollambda{g}{\lambda}$ are connected for all $\lambda$
different from $[2,2]$ and $[4]$. In these two cases, strata decomposes as
\begin{equation*}
 \Hollambda{3}{\lambda} = \mathcal{C}_3^{\lambda,\mathrm{h}} \sqcup \mathcal{C}_3^{\lambda, \mathrm{q}},
\end{equation*}
where $\mathcal{C}_3^{\lambda,\mathrm{h}}$ is the component whose
underlying curves are hyperelliptic and  $\mathcal{C}_3^{\lambda,\mathrm{q}}$
is the component whose underlying curves are not hyperelliptic.
For a more detailed discussion, see \cite{looijengamondello}.

We introduce corresponding loci in $\Holtwo{3}$ which we denote by
$\mathcal{C}_3^{\lambda,\mathrm{h}}[2]$ and $\mathcal{C}_3^{\lambda,\mathrm{q}}[2]$,
respectively, and we denote their closures in $\Holtwo{3}$ by 
$\overline{\mathcal{C}}_3^{\lambda,\mathrm{h}}[2]$ and $\overline{\mathcal{C}}_3^{\lambda,\mathrm{q}}[2]$.
However, these loci are not connected  - we shall shortly see that
the locus $\mathcal{C}_3^{\lambda,\mathrm{h}}[2]$ consists
of $36$ connected components while the locus $\overline{\mathcal{C}}_3^{\lambda,\mathrm{q}}[2]$ consists
of $28$ components. 
This is in sharp contrast with the strata $\Holtwolambda{g}{[2,1^2]}$ and $\Holtwolambda{g}{[3,1]}$
which remain connected after adding the level $2$ structure.

\subsection{Moduli of canonical divisors}

There is a close connection between the moduli spaces $\Holtwo{3}$ and $\mathcal{M}_{3,1}[2]$ which we shall
now explain.
If $\omega$ is a nonzero holomorphic differential on a curve $C$ and $c$ is a nonzero constant, 
then $c \omega$ also is a nonzero holomorphic differential and $c \omega$ has the same
zeros as $\omega$.
We may thus projectivize $\Holtwo{3}$ and the stratification of $\Holtwo{3}$ induces
a stratification of $\mathbb{P}(\Holtwo{3})$.
The space $\mathbb{P}(\Holtwo{3})$ parametrizes genus $3$ curves with level $2$ structure marked
with a canonical divisor.

Now consider the locus $\hyptot \subset \mathcal{M}_{3,1}[2]$.
Let $P$ be the marked point $P$ of a hyperelliptic curve $C$.
There is then a unique canonical divisor containing $P$ in its support.
If $P$ is a Weierstrass point, then this divisor is namely $4P$
and if $P$ is not a Weierstrass point, then this divisor is $2P+2i(P)$
where $i$ denotes the hyperelliptic involution.
We thus see that
\begin{equation*}
 \mathbb{P}(\overline{\mathcal{C}}_3^{[2,2],\mathrm{h}}[2]) \cong
 \langle i \rangle \setminus \hyptot,
\end{equation*}
where $\langle i \rangle$ is the group generated by the hyperelliptic involution.

We now instead consider the locus $\mathcal{Q}_1[2] \subset \mathcal{M}_{3,1}[2]$.
If $P$ is the marked point of a plane quartic curve $C$ we may naturally define
a canonical divisor on $C$ as $D=T_PC \cdot C$. If $P$ is not a genuine bitangent point
(i.e. a bitangent point which is not a hyperflex point), $P$ is the unique
point giving the canonical divisor $D$. We thus have isomorphisms
\begin{align*}
 \mathcal{Q}_{\mathrm{ord}}[2] & \cong \mathbb{P}( \Holtwolambda{3}{[2,1^2]}), \\
 \mathcal{Q}_{\mathrm{flx}}[2] & \cong \mathbb{P}( \Holtwolambda{3}{[3,1]}), \\
 \mathcal{Q}_{\mathrm{hfl}}[2] & \cong \mathbb{P}(\mathcal{C}_3^{[4],\mathrm{q}}[2]), \\
 \mathcal{Q}_{\overline{\mathrm{ord}}}[2] & \cong \mathbb{P}( \Holtwolambdabar{3}{[2,1^2]}) .
\end{align*}
However, if $P$ is a genuine bitangent point we have that $D=T_PC \cdot C= 2P + 2Q$ for
some point $Q \neq P$.
Thus both $P$ and $Q$ give the same canonical divisor $D$.
Let $\beta$ be the involution on $\mathcal{Q}_{\mathrm{btg}}[2]$ sending
a curve marked with a bitangent point to the same curve marked with the 
other point sharing the same bitangent line. Then
\begin{align*}
 \langle \beta \rangle \setminus \mathcal{Q}_{\mathrm{btg}}[2] & \cong \mathbb{P}(\mathcal{C}_3^{[2,2],\mathrm{q}}[2]), \\
 \langle \beta \rangle \setminus \mathcal{Q}_{\overline{\mathrm{btg}}}[2] & \cong \mathbb{P}(\overline{\mathcal{C}}_3^{[2,2],\mathrm{q}}[2]),
\end{align*}
where $\langle \beta \rangle$ is the group generated by $\beta$.

We have that
\begin{equation*}
  \mathcal{Q}_{\mathrm{btg}}[2] \cong \{\pm1 \} \setminus \coprod_{e \in \mathscr{E}} T_{E_6}(e)
\end{equation*}
and $\beta$ acts by sending $\chi \in T_{E_6}(e)$ to $\chi^{-1} \in T_{E_6}(e)$.
Recall that $W(E_7) \cong \mathrm{Sp}(6,\mathbb{F}_2) \times \mathbb{Z}_2$.
The subgroup $W(E_6) \subset W(E_7)$ is entirely contained in $\mathrm{Sp}(6,\mathbb{F}_2)$
and we may therefore identify the group generated by $W(E_6)$ and $\beta$ with $W(E_6) \times \mathbb{Z}_2$.
In order to compute the cohomology of $\mathbb{P}(\mathcal{C}_3^{[2,2],\mathrm{q}}[2])$ we thus
want to compute the $W(E_6) \times \mathbb{Z}_2$-equivariant cohomology of $T_{E_6}$, induce up to
$W(E_7)$ and then take $\{\pm 1\}$-invariants
\begin{equation*}
 H^i(\mathbb{P}(\mathcal{C}_3^{[2,2],\mathrm{q}}[2])) = \mathrm{Ind}_{W(E_6) \times \mathbb{Z}_2}^{W(E_7)}\left( H^i(T_{E_6}) \right)^{\{ \pm 1\}}.
\end{equation*}
Using the results from \cite{bergvalltor}, this computation is straightforward. We present
the result in Table~\ref{Ztable}. In order to obtain the cohomology of $\mathbb{P}(\overline{\mathcal{C}}_3^{[2,2],\mathrm{q}}[2])$
we use the $\mathrm{Sp}(6,\mathbb{F}_2)$-equivariant exact sequence of mixed Hodge structures 
\begin{equation*}
 0 \to H^i(\mathbb{P}(\overline{\mathcal{C}}_3^{[2,2],\mathrm{q}}[2])) 
 \to H^i(\mathbb{P}(\mathcal{C}_3^{[2,2],\mathrm{q}}[2])) 
 \to H^{i-1}(\mathbb{P}(\mathcal{C}_3^{[4],\mathrm{q}}[2]))(-1) \to 0.
\end{equation*}
The result is given in Table~\ref{Zcltable}.

\subsection{Cohomology of moduli spaces of Abelian differentials}

Before we conclude this section we explain how to obtain the cohomology of the non-projectivized spaces
from the cohomology of their projectivized counterparts.

\begin{prop}
 The cohomology of $\Holtwolambda{3}{[2,1^2]}$ is given by
 \begin{equation*}
  H^i(\Holtwolambda{3}{[2,1^2]}) = H^i(\mathbb{P}(\Holtwolambda{3}{[2,1^2]})) \oplus H^{i-2}(\mathbb{P}(\Holtwolambda{3}{[2,1^2]}))(-1).
 \end{equation*}
\end{prop}

\begin{proof}
 The moduli space $\Holtwolambda{3}{[2,1^2]}$ is a $\Pn{1}$-fibration over $\mathbb{P}(\Holtwolambda{3}{[2,1^2]})$
 and the corresponding Leray-Serre spectral sequence degenerates at the second page.
 One then obtains the result by reading off the diagonals.
\end{proof}

Completely analogous arguments gives the following result.

\begin{prop}
 \begin{align*}
  H^i( \Holtwolambda{3}{[3,1]}) & = H^i(\mathbb{P}( \Holtwolambda{3}{[3,1]})) \oplus H^{i-2}(\mathbb{P}( \Holtwolambda{3}{[3,1]}))(-1) \\
  H^i(\mathcal{C}_3^{[2,2],\mathrm{q}}[2]) & = H^i(\mathbb{P}(\mathcal{C}_3^{[2,2],\mathrm{q}}[2])) \oplus H^{i-2}(\mathbb{P}(\mathcal{C}_3^{[2,2],\mathrm{q}}[2]))(-1) \\
  H^i(\mathcal{C}_3^{[4],\mathrm{q}}[2]) & = H^i(\mathbb{P}(\mathcal{C}_3^{[4],\mathrm{q}}[2])) \oplus H^{i-2}(\mathbb{P}(\mathcal{C}_3^{[4],\mathrm{q}}[2]))(-1) \\
  H^i(\Holtwolambdabar{3}{[2,1^2]}) & = H^i(\mathbb{P}( \Holtwolambdabar{3}{[2,1^2]})) \oplus H^{i-2}(\mathbb{P}( \Holtwolambdabar{3}{[2,1^2]}))(-1) \\
  H^i(\overline{\mathcal{C}}_3^{[2,2],\mathrm{q}}[2]) & = H^i(\mathbb{P}(\overline{\mathcal{C}}_3^{[2,2],\mathrm{q}}[2])) \oplus H^{i-2}(\mathbb{P}(\overline{\mathcal{C}}_3^{[2,2],\mathrm{q}}[2]))(-1)
 \end{align*}
\end{prop}

\section{Plain plane quartics}
\label{quarticsec}
We now return to the moduli space $\mathcal{Q}[2]$ and compute its cohomology as
a representation of $\mathrm{Sp}(6,\mathbb{F}_2)$.
A step in this direction was taken in \cite{bergvallpts} where
the cohomology was computed as a representation of the symmetric group
on $7$ elements (a subgroup of index $288$ in $\mathrm{Sp}(6,\mathbb{F}_2)$
which can be thought of as the stabilizer of an unordered Aronhold set of bitangents).
We only reproduce the Poincaré polynomial here and refer to the original article
for the full result.

\begin{prop}[\cite{bergvallpts}]
 The Poincaré polynomial of $\Qt$ is
 \begin{equation*}
  P(\Qt,t) = 1 + 35t + 490t^2 + 3485t^3 + 13174t^4 + 24920t^5 + 18375t^6.
 \end{equation*}
\end{prop}

In order to continue the pursuit of the full structure as a $\mathrm{Sp}(6,\mathbb{F}_2)$-representation
we shall relate $\Qt$ to some of the spaces that have occurred elsewhere in the paper.

\begin{lem}
 The cohomology group $\derham{i}{\Qt}$ is a subrepresentation
  of the $\symp{6,\mathbb{F}_2}$-representation $\derham{i}{\Qbtgcl}$.
  In particular, it is pure of Tate type $(i,i)$.
\end{lem}

\begin{proof}
 The forgetful morphism
 \begin{equation*}
  f: \Qbtgcl \to \Qt,
 \end{equation*}
 is finite so the map 
 \begin{equation*}
  f_{!} \circ f^* : H^i(\Qbtgcl) \to H^i(\Qbtgcl),
 \end{equation*}
 is multiplication with $\mathrm{deg}(f) = 56$.
 Thus, since we are using cohomology with rational coefficients, the map
 \begin{equation*}
  f^*: \derham{i}{\Qt} \to \derham{i}{\Qbtgcl}
 \end{equation*}
 is injective.
\end{proof}

Unfortunately, the cohomology of $\Qbtgcl$ is much too large in comparison 
with the cohomology of $\Qt$ for the above lemma to give any clues as it stands.
The cohomology of $\Qflx$ is however much smaller. To make the comparison, the following lemma shall be useful.

\begin{lem}[Looijenga, \cite{looijenga}]
 \label{gysinlemma}
 Let $X$ be a variety of pure dimension and let $Y \subset X$ be a hypersurface.
 Then there is a Gysin exact sequence of mixed Hodge structures
 \begin{equation*}
  \cdots \to \derham{k-2}{Y}(\text{-}1) \to \derham{k}{X} \to \derham{k}{X\setminus Y} \to \derham{k-1}{Y}(\text{-}1) \to \cdots  
 \end{equation*}
 \end{lem}
 
 \begin{prop}
  The cohomology group $\derham{i}{\Qt}$ is a subrepresentation
 of the $\symp{6,\mathbb{F}_2}$-representation $\derham{i}{\Qflx}$.
 \end{prop}

 \begin{proof}
  Let $X=\Qflx \sqcup \Qhflx$ and let $Y=\Qhflx$ and apply Lemma~\ref{gysinlemma} to
  see that $H^i(X)$ consists of one part of Tate type $(i,i)$ coming from $H^i(\Qflx)$
  and one part of Tate type $(i-1,i-1)$ coming from $\Qhflx$.
  
  The morphism $f: X \to \Qt$ forgetting the marked point is finite of degree $24$
  so the map 
 \begin{equation*}
  f_{!} \circ f^* : H^i(X) \to H^i(\Qbtgcl),
 \end{equation*}
 is multiplication with $24$. In particular
 \begin{equation*}
  f^*: \derham{i}{\Qt} \to \derham{i}{X}
 \end{equation*}
 is injective. But $\derham{i}{\Qt}$ is pure of Tate type $(i,i)$
 so the image of $f^*$ must lie inside the $(i,i)$-part of $H^i(X)$
 which we can identify with a subspace of $H^i(\Qflx)$ by the above.
 \end{proof}
 
 One could now hope that knowing that $H^i(\Qt)$ is a subrepresentation
 of \linebreak[4]$H^i(\Qflx)$ together with the information about how this representation
 restricts to $S_7$ from \cite{bergvallpts} would determine $H^i(\Qt)$ as a
 representation of $\mathrm{Sp}(6,\mathbb{F}_2)$.
 This is the case for $i=0,1,2$ and $3$ but not for $i=4,5$ and $6$.
 For instance, in the case $i=4$ there are $1039$ representations that satisfy
 these conditions.
 
 Observe that the space $\mathbb{P}(\overline{\mathcal{C}}_3^{[2,2],\mathrm{q}}[2])$ parametrizes
 plane quartics with level $2$ structure marked with a bitangent line (here we also allow hyperflex lines as bitangent lines).
 We consider a level $2$ structure on a quartic $C$ as an ordered Aronhold set $(\theta_1, \ldots, \theta_7)$ of odd theta characteristics.
 The odd theta characteristics not in the Aronhold set are of the form
 \begin{equation*}
  \theta_{i,j}:=\sum_{k=1}^7 \theta_k-\theta_i-\theta_j, \quad 1 \leq i < j \leq 7.
 \end{equation*}
 We define 
 \begin{equation*}
  \mathcal{B}_{i} \subset \mathbb{P}(\overline{\mathcal{C}}_3^{[2,2],\mathrm{q}}[2])
 \end{equation*}
 as the subset of points such that the marked bitangent induces the $i$'th theta characteristic of
 the ordered Aronhold set.
 Similarly, we define $\mathcal{B}_{i,j}$ as the subset where the marked bitangent induces
 the theta characteristic $\theta_{i,j}$.
 
 Let $I$ be any subset of $\{1, \ldots, 7\}$ of size $1$ or $2$.
 The spaces $\mathcal{B}_{I}$ are all isomorphic
 and
 \begin{equation*}
  \mathbb{P}(\overline{\mathcal{C}}_3^{[2,2],\mathrm{q}}[2]) = \coprod_I \mathcal{B}_I
 \end{equation*}
 where the union is over all such subsets $I$. Moreover, we have an isomorphism
 \begin{equation*}
  \Qt \to \mathcal{B}_1
 \end{equation*}
 sending the class of a plane quartic with level $2$ structure, where we think of the level structure
 as an ordered Aronhold set of bitangent lines, to the class of the same curve with the same level $2$ structure
 marked with the first bitangent of the Aronhold set. 
 
 We now rephrase the above slightly.
 It is well-known that the stabilizer $\mathrm{Stab}(b)\subset \mathrm{Sp}(6,\mathbb{F}_2)$ of a bitangent line $b$
 is isomorphic to $W(E_6)$. Let $\mathscr{S}$ denote the quotient set $\mathrm{Sp}(6,\mathbb{F}_2)/W(E_6)$
 and let $[\sigma]$ denote the class of $\sigma \in \mathrm{Sp}(6,\mathbb{F}_2)$ in $\mathscr{S}$. 
 If we now let $X_{[\sigma]}=\Qt$ we have
 \begin{equation*}
  \mathbb{P}(\overline{\mathcal{C}}_3^{[2,2],\mathrm{q}}[2]) \cong \coprod_{\sigma \in \mathscr{S}} X_{[\sigma]}
 \end{equation*}
 and $\mathrm{Sp}(6,\mathbb{F}_2)$ acts transitively on the set of connected components of
 $\mathbb{P}(\overline{\mathcal{C}}_3^{[2,2],\mathrm{q}}[2])$. In particular, we find that
 \begin{equation}
 \label{indreseq}
  H^i(\mathbb{P}(\overline{\mathcal{C}}_3^{[2,2],\mathrm{q}}[2])) = 
  \mathrm{Ind}_{W(E_6)}^{\mathrm{Sp}(6,\mathbb{F}_2)}
  \mathrm{Res}_{W(E_6)}^{\mathrm{Sp}(6,\mathbb{F}_2)}
  \left( H^i(\Qt) \right). 
 \end{equation}
 
 We now know how the cohomology groups of $\Qt$ restricts to $S_7$, and we know how they
 relate both to the cohomology groups of $\Qflx$ and the cohomology groups of 
 $\mathbb{P}(\overline{\mathcal{C}}_3^{[2,2],\mathrm{q}}[2])$.
 This information is enough to finally determine $H^i(\Qt)$ in all degrees.
 The results are presented in Table~\ref{Qtable}.
 We remark that even though it may seem plausible at first glance,
 Equation~\ref{indreseq} does not determine the cohomology of $\Qt$ by itself.

\section{The moduli space $\mathcal{M}_3[2]$}
\label{mthreesec}
We now consider the cohomology of the moduli space $\mathcal{M}_3[2]$ 
of genus $3$ curves with level $2$ structure.
By applying Lemma~\ref{gysinlemma} to the decomposition
$\mathcal{M}_3[2] = \Qt \sqcup \Htwo{3}$ we obtain the exact sequence of mixed Hodge structures
\begin{equation*}
 \cdots \to \derham{k-2}{\Htwo{3}}(\text{-}1) \to \derham{k}{\mathcal{M}_3[2]} \to \derham{k}{\Qt} \to \derham{k-1}{\Htwo{3}}(\text{-}1) \to \cdots
\end{equation*}
Since both $\derham{k}{\Htwo{3}}(\text{-}1)$ and $\derham{k}{\Qt}(\text{-}1)$ are pure
of Tate type $(k,k)$, the above sequence decomposes into sequences
\begin{equation*}
0 \to W_{2k}\derham{k}{\mathcal{M}_3[2]} \to \derham{k}{\Qt} \to \derham{k-1}{\Htwo{3}}(\text{-}1) \to W_{2k}\derham{k+1}{\mathcal{M}_3[2]} \to 0,
\end{equation*}
where $W_{2k}\derham{k}{\mathcal{M}_3[2]}$ denotes the weight $2k$ part of $\derham{k}{\mathcal{M}_3[2]}$.
Taking $k=0$ we obtain that $\derham{0}{\mathcal{M}_3[2]} = \derham{0}{\Qt} = \mathbb{Q}$ which is
not very surprising.

For $k=1$ and $k=2$ we have the results of Hain \cite{hain} and Putman \cite{putman}
that $H^k(\mathcal{M}_3[2]) \cong H^k(\mathcal{M}_3)$.
Combined with the results of Looijenga \cite{looijenga} we have that
$H^1(\mathcal{M}_3[2])=0$ and $H^2(\mathcal{M}_3[2])=\mathbb{Q}$ with Tate type $(1,1)$
and it is reassuring to see that this is indeed compatible with the above
sequence.

Given the complexity of the cohomology of $\Qt$ and $\Htwo{3}$, the cohomology of $\mathcal{M}_3[2]$
is surprisingly simple in low degrees. This phenomenon will not prevail in all degrees though.
For instance, taking $k=7$ the above sequence gives that $\mathrm{dim}(H^7(\mathcal{M}_3[2])) \geq 7680$.
This bound is in fact far from optimal, as Fullarton and Putman \cite{fullartonputman} recently
have shown that $\mathrm{dim}(H^7(\mathcal{M}_3[2])) \geq 11520$ via completely different methods.
In particular, we see that the cohomology of $\mathcal{M}_3[2]$ is not the smallest possible
fitting in a four term exact sequence of the above type. We also remark that we get an upper bound
$\mathrm{dim}(H^7(\mathcal{M}_3[2])) \leq \mathrm{dim}(H^5(\Htwo{3})) = 25920$.

\subsection{The weighted Euler characteristic}

Recall that the Poincaré-Serre polynomial of a variety $X$ is defined as
\begin{equation*}
 PS(X,t,u) = \sum_{i,j \geq 0} W_jH^i(X)t^iu^j
\end{equation*}
where the sum is taken in the Grothendieck ring of vector spaces.
By setting $t=-1$ in $PS(X,t,u)$ we obtain the weighted Euler characteristic $\mathrm{Eul}(X,u)$.
Using the above exact sequence the weighted Euler characteristic of $\mathcal{M}_3[2]$ can easily be 
deduced from Table~\ref{Qtable} and
Table~\ref{Hyptable}.


\newpage
\section{Tables}
\label{tablesec}
In the tables below we present the cohomology of various spaces occurring throughout the
paper as representations of the group $\symp{6,\ftwo}$.
The rows of the tables represent the cohomology groups and the columns correspond
to irreducible representations of $\symp{6,\ftwo}$.
Thus, a number $n$ in the row indexed by $H^i$ and column indexed by $\phi$ means that
the irreducible representation $\phi$ occurs with multiplicity $n$ in $H^i$.

The irreducible representations are denoted as $\phi_{dx}$ where the subscripts follow the conventions of
\cite{conwayetal}, i.e. $d$ denotes the dimension of the representation and the letter $x$
denotes is used to distinguish different representations of the same dimension.
The letters used here are the same as in \cite{conwayetal}.

\newpage

\begin{table}[htbp]
\begin{equation*}
\resizebox{0.8\textwidth}{!}{$
\begin{array}{r|rrrrrrrrrr} 
\, & \phi_{1a} & \phi_{7a} & \phi_{15a} & \phi_{21a} & \phi_{21b} & \phi_{27a} & \phi_{35a} & \phi_{35b} & \phi_{56a} & \phi_{70a} \\
\hline
H^0 & 1 & 0 & 0 & 0 & 0 & 0 & 0 & 0 & 0 & 0 \\
H^1 & 1 & 0 & 0 & 0 & 0 & 1 & 0 & 1 & 0 & 0 \\
H^2 & 0 & 0 & 0 & 1 & 0 & 2 & 0 & 2 & 0 & 0 \\
H^3 & 0 & 0 & 0 & 3 & 0 & 3 & 0 & 3 & 0 & 0 \\
H^4 & 0 & 0 & 0 & 7 & 0 & 8 & 2 & 9 & 0 & 5 \\
H^5 & 0 & 0 & 3 & 17 & 2 & 25 & 16 & 30 & 11 & 30 \\
H^6 & 2 & 4 & 18 & 34 & 19 & 50 & 45 & 63 & 53 & 86 \\
H^7 & 2 & 8 & 19 & 34 & 25 & 43 & 47 & 52 & 74 & 101 \\
\hline
\, & \phi_{84a} & \phi_{105a} & \phi_{105b} & \phi_{105c} & \phi_{120a} & \phi_{168a} & \phi_{189a} & \phi_{189b} & \phi_{189c} & \phi_{210a} \\
\hline
H^0 & 0 & 0 & 0 & 0 & 0 & 0 & 0 & 0 & 0 & 0 \\
H^1 & 0 & 0 & 0 & 0 & 0 & 0 & 0 & 0 & 0 & 0 \\
H^2 & 0 & 0 & 1 & 0 & 2 & 1 & 0 & 0 & 0 & 0 \\
H^3 & 1 & 0 & 7 & 2 & 9 & 7 & 5 & 0 & 0 & 4 \\
H^4 & 9 & 1 & 27 & 14 & 33 & 36 & 33 & 5 & 7 & 32 \\
H^5 & 50 & 29 & 78 & 63 & 99 & 128 & 125 & 61 & 73 & 128 \\
H^6 & 127 & 113 & 160 & 154 & 194 & 267 & 277 & 215 & 233 & 295 \\
H^7 & 117 & 137 & 159 & 147 & 185 & 249 & 276 & 255 & 251 & 307 \\
\hline
\, & \phi_{210b} & \phi_{216a} & \phi_{280a} & \phi_{280b} & \phi_{315a} & \phi_{336a} & \phi_{378a} & \phi_{405a} & \phi_{420a} & \phi_{512a} \\
\hline
H^0 & 0 & 0 & 0 & 0 & 0 & 0 & 0 & 0 & 0 & 0 \\
H^1 & 0 & 0 & 0 & 0 & 0 & 0 & 0 & 0 & 0 & 0 \\
H^2 & 2 & 0 & 0 & 2 & 0 & 0 & 0 & 0 & 0 & 0 \\
H^3 & 13 & 1 & 0 & 9 & 0 & 2 & 2 & 11 & 7 & 6 \\
H^4 & 51 & 13 & 19 & 47 & 21 & 33 & 33 & 73 & 61 & 61 \\
H^5 & 157 & 99 & 126 & 191 & 141 & 179 & 188 & 268 & 258 & 290 \\
H^6 & 326 & 287 & 351 & 427 & 393 & 456 & 498 & 588 & 598 & 710 \\
H^7 & 313 & 296 & 388 & 404 & 441 & 468 & 533 & 598 & 602 & 731
\end{array}$}
\end{equation*}
\caption{The cohomology of $\Qord$ as a representation of $\symp{6,\ftwo}$.}
\label{Qordtable}
\begin{equation*}
\resizebox{0.8\textwidth}{!}{$
\begin{array}{r|rrrrrrrrrr} 
\, & \phi_{1a} & \phi_{7a} & \phi_{15a} & \phi_{21a} & \phi_{21b} & \phi_{27a} & \phi_{35a} & \phi_{35b} & \phi_{56a} & \phi_{70a} \\
\hline
H^0 &  1&0&0&0&0&0&0&0&0&0\\ 
H^1 & 0&0&0&0&0&1&0&1&0&0\\ 
H^2 & 0&0&0&0&0&1&0&1&0&0\\
H^3 & 0&0&0&2&0&2&0&2&0&0\\ 
H^4 & 0&0&0&5&0&6&1&6&0&5\\ 
H^5 & 0&0&3&10&2&15&11&18&9&20\\ 
H^6 & 1&2&7&13&8&17&18&22&23&35 \\
\hline
\, & \phi_{84a} & \phi_{105a} & \phi_{105b} & \phi_{105c} & \phi_{120a} & \phi_{168a} & \phi_{189a} & \phi_{189b} & \phi_{189c} & \phi_{210a} \\
\hline
H^0 & 0&0&0&0&0&0&0&0&0&0\\ 
H^1 & 0&0&0&0&0&0&0&0&0&0\\
H^2 & 0&0&1&0&2&1&0&0&0&2\\ 
H^3 & 1&0&6&2&6&6&4&0&0&10\\ 
H^4 & 8&1&18&11&24&27&25&5&7&35\\ 
H^5 & 32&22&45&39&59&76&77&45&51&93\\ 
H^6 & 47&47&60&58&69&98&104&88&92&120 \\
\hline
\, & \phi_{210b} & \phi_{216a} & \phi_{280a} & \phi_{280b} & \phi_{315a} & \phi_{336a} & \phi_{378a} & \phi_{405a} & \phi_{420a} & \phi_{512a} \\
\hline
H^0 & 0&0&0&0&0&0&0&0&0&0\\ 
H^1 & 0&0&0&0&0&0&0&0&0&0\\ 
H^2 & 0&0&0&2&0&0&0&0&0&0\\ 
H^3 & 4&1&0&7&0&2&2&10&7&6\\ 
H^4 & 25&12&17&36&19&27&29&55&47&50\\ 
H^5 & 80&68&83&118&96&116&124&164&160&184\\ 
H^6 & 111&111&140&157&155&175&193&221&226&272
\end{array}$}
\end{equation*}
\caption{The cohomology of $\Qflx$ as a representation of $\symp{6,\ftwo}$.}
\label{Qflxtable}
\end{table}

 \begin{table}[htbp]
\begin{equation*}
\resizebox{0.8\textwidth}{!}{$
\begin{array}{r|rrrrrrrrrr} 
\, & \phi_{1a} & \phi_{7a} & \phi_{15a} & \phi_{21a} & \phi_{21b} & \phi_{27a} & \phi_{35a} & \phi_{35b} & \phi_{56a} & \phi_{70a} \\
\hline
H^0 & 1 & 0 & 0 & 0 & 0 & 1 & 0 & 0 & 0 & 0 \\
H^1 & 1 & 0 & 0 & 1 & 0 & 3 & 0 & 2 & 0 & 0 \\
H^2 & 0 & 0 & 0 & 2 & 0 & 4 & 0 & 5 & 0 & 0 \\
H^3 & 0 & 0 & 0 & 6 & 0 & 7 & 2 & 10 & 0 & 4 \\
H^4 & 0 & 0 & 3 & 17 & 2 & 20 & 15 & 25 & 11 & 30 \\
H^5 & 1 & 4 & 14 & 30 & 17 & 41 & 39 & 49 & 51 & 80 \\
H^6 & 2 & 7 & 18 & 25 & 22 & 35 & 39 & 44 & 60 & 78 \\
\hline
\, & \phi_{84a} & \phi_{105a} & \phi_{105b} & \phi_{105c} & \phi_{120a} & \phi_{168a} & \phi_{189a} & \phi_{189b} & \phi_{189c} & \phi_{210a} \\
\hline
H^0 & 0 & 0 & 0 & 0 & 0 & 0 & 0 & 0 & 0 & 0 \\
H^1 & 0 & 0 & 1 & 0 & 2 & 1 & 0 & 0 & 0 & 1 \\
H^2 & 1 & 0 & 6 & 2 & 9 & 8 & 4 & 0 & 0 & 11 \\
H^3 & 10 & 1 & 24 & 15 & 30 & 35 & 31 & 4 & 7 & 48 \\
H^4 & 46 & 27 & 72 & 60 & 89 & 114 & 118 & 58 & 69 & 146 \\
H^5 & 105 & 105 & 140 & 129 & 169 & 229 & 243 & 198 & 206 & 282 \\
H^6 & 98 & 112 & 124 & 119 & 143 & 197 & 215 & 207 & 207 & 244 \\
\hline
\, & \phi_{210b} & \phi_{216a} & \phi_{280a} & \phi_{280b} & \phi_{315a} & \phi_{336a} & \phi_{378a} & \phi_{405a} & \phi_{420a} & \phi_{512a} \\
\hline
H^0 & 0 & 0 & 0 & 0 & 0 & 0 & 0 & 0 & 0 & 0 \\
H^1 & 0 & 0 & 0 & 1 & 0 & 0 & 0 & 0 & 0 & 0 \\
H^2 & 3 & 1 & 0 & 10 & 0 & 2 & 1 & 9 & 5 & 5 \\
H^3 & 27 & 14 & 16 & 49 & 18 & 32 & 30 & 67 & 56 & 58 \\
H^4 & 120 & 92 & 120 & 173 & 134 & 168 & 179 & 252 & 242 & 272 \\
H^5 & 265 & 250 & 319 & 364 & 360 & 401 & 447 & 522 & 526 & 629 \\
H^6 & 241 & 243 & 309 & 323 & 349 & 375 & 423 & 463 & 477 & 578
\end{array}$}
\end{equation*}
\caption{The cohomology of $\Qbtg$ as a representation of $\symp{6,\ftwo}$.}
\label{Qbtgtable}
\begin{equation*}
\resizebox{0.8\textwidth}{!}{$
\begin{array}{r|rrrrrrrrrr} 
\, & \phi_{1a} & \phi_{7a} & \phi_{15a} & \phi_{21a} & \phi_{21b} & \phi_{27a} & \phi_{35a} & \phi_{35b} & \phi_{56a} & \phi_{70a} \\
\hline
H^0 & 1&0&0&0&0&1&0&0&0&0 \\
H^1 & 0&0&0&0&0&1&0&2&0&0 \\ 
H^2 & 0&0&0&0&0&1&0&2&0&0 \\ 
H^3 & 0&0&0&4&0&3&1&3&0&3 \\
H^4 & 0&0&1&6&1&7&5&7&6&13 \\
H^5 & 0&1&4&5&4&8&9&11&10&14 \\
\hline
\, & \phi_{84a} & \phi_{105a} & \phi_{105b} & \phi_{105c} & \phi_{120a} & \phi_{168a} & \phi_{189a} & \phi_{189b} & \phi_{189c} & \phi_{210a} \\
\hline
H^0 & 0&0&0&0&0&0&0&0&0&0 \\ 
H^1 & 0&0&1&0&1&1&0&0&0&1 \\ 
H^2 & 1&0&3&2&4&4&2&0&0&6 \\ 
H^3 & 5&1&10&7&13&15&16&3&5&21 \\ 
H^4 & 15&12&25&19&31&39&41&26&26&49 \\ 
H^5 & 23&22&26&27&31&45&46&40&44&53 \\
\hline
\, & \phi_{210b} & \phi_{216a} & \phi_{280a} & \phi_{280b} & \phi_{315a} & \phi_{336a} & \phi_{378a} & \phi_{405a} & \phi_{420a} & \phi_{512a} \\
\hline
H^0 & 0&0&0&0&0&0&0&0&0&0 \\ 
H^1 & 0&0&0&1&0&0&0&0&0&0 \\ 
H^2 & 2&1&0&7&0&2&1&6&4&4 \\ 
H^3 & 15&8&10&21&12&17&19&33&29&31 \\ 
H^4 & 44&34&48&57&54&60&68&90&85&99 \\ 
H^5 & 49&53&62&74&69&81&86&96&102&122
\end{array}$}
\end{equation*}
\caption{The cohomology of $\Qhflx$ as a representation of $\symp{6,\ftwo}$.}
\label{Qhflxtable}
\end{table}

 \begin{table}[htbp]
\begin{equation*}
\resizebox{0.8\textwidth}{!}{$
\begin{array}{r|rrrrrrrrrr} 
\, & \phi_{1a} & \phi_{7a} & \phi_{15a} & \phi_{21a} & \phi_{21b} & \phi_{27a} & \phi_{35a} & \phi_{35b} & \phi_{56a} & \phi_{70a} \\
\hline
H^0 & 1 & 0 & 0 & 0 & 0 & 0 & 0 & 0 & 0 & 0 \\
H^1 & 0 & 0 & 0 & 0 & 0 & 1 & 0 & 1 & 0 & 0 \\
H^2 & 0 & 0 & 0 & 1 & 0 & 1 & 0 & 1 & 0 & 0 \\
H^3 & 0 & 0 & 0 & 3 & 0 & 2 & 0 & 2 & 0 & 0 \\
H^4 & 0 & 0 & 0 & 5 & 0 & 6 & 2 & 7 & 0 & 5 \\
H^5 & 0 & 0 & 3 & 12 & 2 & 19 & 15 & 24 & 11 & 25 \\
H^6 & 2 & 4 & 15 & 24 & 17 & 35 & 34 & 45 & 44 & 66 \\
H^7 & 1 & 6 & 12 & 21 & 17 & 26 & 29 & 30 & 51 & 66 \\
\hline
\, & \phi_{84a} & \phi_{105a} & \phi_{105b} & \phi_{105c} & \phi_{120a} & \phi_{168a} & \phi_{189a} & \phi_{189b} & \phi_{189c} & \phi_{210a} \\
\hline
H^0 & 0 & 0 & 0 & 0 & 0 & 0 & 0 & 0 & 0 & 0 \\
H^1 & 0 & 0 & 0 & 0 & 0 & 0 & 0 & 0 & 0 & 0 \\
H^2 & 0 & 0 & 1 & 0 & 2 & 1 & 0 & 0 & 0 & 0 \\
H^3 & 1 & 0 & 6 & 2 & 7 & 6 & 5 & 0 & 0 & 2 \\
H^4 & 8 & 1 & 21 & 12 & 27 & 30 & 29 & 5 & 7 & 22 \\
H^5 & 42 & 28 & 60 & 52 & 75 & 101 & 100 & 56 & 66 & 93 \\
H^6 & 95 & 91 & 115 & 115 & 135 & 191 & 200 & 170 & 182 & 202 \\
H^7 & 70 & 90 & 99 & 89 & 116 & 151 & 172 & 167 & 159 & 187 \\
\hline
\, & \phi_{210b} & \phi_{216a} & \phi_{280a} & \phi_{280b} & \phi_{315a} & \phi_{336a} & \phi_{378a} & \phi_{405a} & \phi_{420a} & \phi_{512a} \\
\hline
H^0 & 0 & 0 & 0 & 0 & 0 & 0 & 0 & 0 & 0 & 0 \\
H^1 & 0 & 0 & 0 & 0 & 0 & 0 & 0 & 0 & 0 & 0 \\
H^2 & 2 & 0 & 0 & 2 & 0 & 0 & 0 & 0 & 0 & 0 \\
H^3 & 13 & 1 & 0 & 7 & 0 & 2 & 2 & 11 & 7 & 6 \\
H^4 & 47 & 12 & 19 & 40 & 21 & 31 & 31 & 63 & 54 & 55 \\
H^5 & 132 & 87 & 109 & 155 & 122 & 152 & 159 & 213 & 211 & 240 \\
H^6 & 246 & 219 & 268 & 309 & 297 & 340 & 374 & 424 & 438 & 526 \\
H^7 & 202 & 185 & 248 & 247 & 286 & 293 & 340 & 377 & 376 & 459
\end{array}$}
\end{equation*}
\caption{The cohomology of $\Qordcl$ as a representation of $\symp{6,\ftwo}$.}
\label{Qordcltable}
\begin{equation*}
\resizebox{0.8\textwidth}{!}{$
\begin{array}{r|rrrrrrrrrr} 
\, & \phi_{1a} & \phi_{7a} & \phi_{15a} & \phi_{21a} & \phi_{21b} & \phi_{27a} & \phi_{35a} & \phi_{35b} & \phi_{56a} & \phi_{70a} \\
\hline
H^0 & 1 & 0 & 0 & 0 & 0 & 1 & 0 & 0 & 0 & 0 \\
H^1 & 0 & 0 & 0 & 1 & 0 & 2 & 0 & 2 & 0 & 0 \\
H^2 & 0 & 0 & 0 & 2 & 0 & 3 & 0 & 3 & 0 & 0 \\
H^3 & 0 & 0 & 0 & 6 & 0 & 6 & 2 & 8 & 0 & 4 \\
H^4 & 0 & 0 & 3 & 13 & 2 & 17 & 14 & 22 & 11 & 27 \\
H^5 & 1 & 4 & 13 & 24 & 16 & 34 & 34 & 42 & 45 & 67 \\
H^6 & 2 & 6 & 14 & 20 & 18 & 27 & 30 & 33 & 50 & 64 \\
\hline
\, & \phi_{84a} & \phi_{105a} & \phi_{105b} & \phi_{105c} & \phi_{120a} & \phi_{168a} & \phi_{189a} & \phi_{189b} & \phi_{189c} & \phi_{210a} \\
\hline
H^0 & 0 & 0 & 0 & 0 & 0 & 0 & 0 & 0 & 0 & 0 \\
H^1 & 0 & 0 & 1 & 0 & 2 & 1 & 0 & 0 & 0 & 1 \\
H^2 & 1 & 0 & 5 & 2 & 8 & 7 & 4 & 0 & 0 & 10 \\
H^3 & 9 & 1 & 21 & 13 & 26 & 31 & 29 & 4 & 7 & 42 \\
H^4 & 41 & 26 & 62 & 53 & 76 & 99 & 102 & 55 & 64 & 125 \\
H^5 & 90 & 93 & 115 & 110 & 138 & 190 & 202 & 172 & 180 & 233 \\
H^6 & 75 & 90 & 98 & 92 & 112 & 152 & 169 & 167 & 163 & 191 \\
\hline
\, & \phi_{210b} & \phi_{216a} & \phi_{280a} & \phi_{280b} & \phi_{315a} & \phi_{336a} & \phi_{378a} & \phi_{405a} & \phi_{420a} & \phi_{512a} \\
\hline
H^0 & 0 & 0 & 0 & 0 & 0 & 0 & 0 & 0 & 0 & 0 \\
H^1 & 0 & 0 & 0 & 1 & 0 & 0 & 0 & 0 & 0 & 0 \\
H^2 & 3 & 1 & 0 & 9 & 0 & 2 & 1 & 9 & 5 & 5 \\
H^3 & 25 & 13 & 16 & 42 & 18 & 30 & 29 & 61 & 52 & 54 \\
H^4 & 105 & 84 & 110 & 152 & 122 & 151 & 160 & 219 & 213 & 241 \\
H^5 & 221 & 216 & 271 & 307 & 306 & 341 & 379 & 432 & 441 & 530 \\
H^6 & 192 & 190 & 247 & 249 & 280 & 294 & 337 & 367 & 375 & 456
\end{array}$}
\end{equation*}
\caption{The cohomology of $\Qbtgcl$ as a representation of $\symp{6,\ftwo}$.}
\label{Qbtgcltable}
\end{table}

 \begin{table}[htbp]
\begin{equation*}
\resizebox{0.8\textwidth}{!}{$
\begin{array}{r|rrrrrrrrrr} 
\, & \phi_{1a} & \phi_{7a} & \phi_{15a} & \phi_{21a} & \phi_{21b} & \phi_{27a} & \phi_{35a} & \phi_{35b} & \phi_{56a} & \phi_{70a} \\
\hline
H^0 & 1 & 0 & 0 & 0 & 0 & 1 & 0 & 0 & 0 & 0 \\
H^1 & 1 & 0 & 0 & 0 & 0 & 2 & 0 & 2 & 0 & 0 \\
H^2 & 0 & 0 & 0 & 1 & 0 & 2 & 0 & 4 & 0 & 0 \\
H^3 & 0 & 0 & 0 & 5 & 0 & 4 & 1 & 5 & 0 & 3 \\
H^4 & 0 & 0 & 1 & 11 & 1 & 11 & 8 & 12 & 8 & 20  \\
H^5 & 1 & 2 & 8 & 16 & 9 & 23 & 22 & 28 & 28 & 44 \\
H^6 & 2 & 3 & 14 & 11 & 14 & 21 & 23 & 31 & 26 & 34 \\
\hline
\, & \phi_{84a} & \phi_{105a} & \phi_{105b} & \phi_{105c} & \phi_{120a} & \phi_{168a} & \phi_{189a} & \phi_{189b} & \phi_{189c} & \phi_{210a} \\
\hline
H^0 & 0 & 0 & 0 & 0 & 0 & 0 & 0 & 0 & 0 & 0 \\
H^1 & 0 & 0 & 1 & 0 & 1 & 1 & 0 & 0 & 0 & 1 \\
H^2 & 1 & 0 & 4 & 2 & 5 & 5 & 3 & 0 & 0 & 8 \\
H^3 & 6 & 1 & 14 & 10 & 19 & 20 & 21 & 3 & 5 & 31 \\
H^4 & 23 & 17 & 44 & 31 & 54 & 65 & 70 & 36 & 38 & 86 \\
H^5 & 58 & 58 & 79 & 69 & 93 & 128 & 133 & 109 & 112 & 155 \\
H^6 & 65 & 59 & 61 & 73 & 70 & 110 & 111 & 106 & 121 & 129 \\
\hline
\, & \phi_{210b} & \phi_{216a} & \phi_{280a} & \phi_{280b} & \phi_{315a} & \phi_{336a} & \phi_{378a} & \phi_{405a} & \phi_{420a} & \phi_{512a} \\
\hline
H^0 & 0 & 0 & 0 & 0 & 0 & 0 & 0 & 0 & 0 & 0 \\
H^1 & 0 & 0 & 0 & 1 & 0 & 0 & 0 & 0 & 0 & 0 \\
H^2 & 2 & 1 & 0 & 8 & 0 & 2 & 1 & 6 & 4 & 4 \\
H^3 & 18 & 9 & 11 & 30 & 13 & 21 & 22 & 44 & 37 & 38 \\
H^4 & 75 & 50 & 74 & 95 & 85 & 96 & 108 & 155 & 143 & 161 \\
H^5 & 147 & 137 & 177 & 200 & 199 & 220 & 244 & 288 & 290 & 346 \\
H^6 & 117 & 143 & 157 & 186 & 169 & 207 & 217 & 225 & 253 & 303
\end{array}$}
\end{equation*}
\caption{The cohomology of $\mathbb{P}(\mathcal{C}_3^{[2,2],\mathrm{q}}[2])$ as a representation of $\symp{6,\ftwo}$.}
\label{Ztable}

\begin{equation*}
\resizebox{0.8\textwidth}{!}{$
\begin{array}{r|rrrrrrrrrr} 
\, & \phi_{1a} & \phi_{7a} & \phi_{15a} & \phi_{21a} & \phi_{21b} & \phi_{27a} & \phi_{35a} & \phi_{35b} & \phi_{56a} & \phi_{70a} \\
\hline
H^0 & 1 & 0 & 0 & 0 & 0 & 1 & 0 & 0 & 0 & 0 \\
H^1 & 0 & 0 & 0 & 0 & 0 & 1 & 0 & 2 & 0 & 0 \\
H^2 & 0 & 0 & 0 & 1 & 0 & 1 & 0 & 2 & 0 & 0 \\
H^3 & 0 & 0 & 0 & 5 & 0 & 3 & 1 & 3 & 0 & 3 \\
H^4 & 0 & 0 & 1 & 7 & 1 & 8 & 7 & 9 & 8 & 17 \\
H^5 & 1 & 2 & 7 & 10 & 8 & 16 & 17 & 21 & 22 & 31 \\
H^6 & 2 & 2 & 10 & 6 & 10 & 13 & 14 & 20 & 16 & 20 \\
\hline
\, & \phi_{84a} & \phi_{105a} & \phi_{105b} & \phi_{105c} & \phi_{120a} & \phi_{168a} & \phi_{189a} & \phi_{189b} & \phi_{189c} & \phi_{210a} \\
\hline
H^0 & 0 & 0 & 0 & 0 & 0 & 0 & 0 & 0 & 0 & 0 \\
H^1 & 0 & 0 & 1 & 0 & 1 & 1 & 0 & 0 & 0 & 1 \\
H^2 & 1 & 0 & 3 & 2 & 4 & 4 & 3 & 0 & 0 & 7   \\
H^3 & 5 & 1 & 11 & 8 & 15 & 16 & 19 & 3 & 5 & 25 \\
H^4 & 18 & 16 & 34 & 24 & 41 & 50 & 54 & 33 & 33 & 65 \\
H^5 & 43 & 46 & 54 & 50 & 62 & 89 & 92 & 83 & 86 & 106 \\
H^6 & 42 & 37 & 35 & 46 & 39 & 65 & 65 & 66 & 77 & 76 \\
\hline
\, & \phi_{210b} & \phi_{216a} & \phi_{280a} & \phi_{280b} & \phi_{315a} & \phi_{336a} & \phi_{378a} & \phi_{405a} & \phi_{420a} & \phi_{512a} \\
\hline
H^0 & 0 & 0 & 0 & 0 & 0 & 0 & 0 & 0 & 0 & 0 \\
H^1 & 0 & 0 & 0 & 1 & 0 & 0 & 0 & 0 & 0 & 0 \\
H^2 & 2 & 1 & 0 & 7 & 0 & 2 & 1 & 6 & 4 & 4 \\
H^3 & 16 & 8 & 11 & 23 & 13 & 19 & 21 & 38 & 33 & 34 \\
H^4 & 60 & 42 & 64 & 74 & 73 & 79 & 89 & 122 & 114 & 130 \\
H^5 & 103 & 103 & 129 & 143 & 145 & 160 & 176 & 198 & 205 & 247 \\
H^6 & 68 & 90 & 95 & 112 & 100 & 126 & 131 & 129 & 151 & 181
\end{array}$}
\end{equation*}
\caption{The cohomology of $\mathbb{P}(\overline{\mathcal{C}}_3^{[2,2],\mathrm{q}}[2])$ as a representation of $\symp{6,\ftwo}$.}
\label{Zcltable}
\end{table}

 \begin{table}[htbp]
\begin{equation*}
\resizebox{0.8\textwidth}{!}{$
\begin{array}{r|rrrrrrrrrr} 
\, & \phi_{1a} & \phi_{7a} & \phi_{15a} & \phi_{21a} & \phi_{21b} & \phi_{27a} & \phi_{35a} & \phi_{35b} & \phi_{56a} & \phi_{70a} \\
\hline
H^0 & 1 & 0 & 0 & 0 & 0 & 0 & 0 & 0 & 0 & 0 \\
H^1 & 0 & 0 & 0 & 0 & 0 & 0 & 0 & 1 & 0 & 0 \\
H^2 & 0 & 0 & 0 & 0 & 0 & 0 & 0 & 0 & 0 & 0 \\
H^3 & 0 & 0 & 0 & 1 & 0 & 0 & 0 & 0 & 0 & 0 \\
H^4 & 0 & 0 & 0 & 0 & 0 & 0 & 0 & 0 & 0 & 1 \\
H^5 & 0 & 0 & 0 & 0 & 0 & 1 & 1 & 1 & 0 & 0 \\
H^6 & 1 & 0 & 2 & 0 & 1 & 1 & 1 & 3 & 0 & 0 \\
\hline
\, & \phi_{84a} & \phi_{105a} & \phi_{105b} & \phi_{105c} & \phi_{120a} & \phi_{168a} & \phi_{189a} & \phi_{189b} & \phi_{189c} & \phi_{210a} \\
\hline
H^0 & 0 & 0 & 0 & 0 & 0 & 0 & 0 & 0 & 0 & 0 \\
H^1 & 0 & 0 & 0 & 0 & 0 & 0 & 0 & 0 & 0 & 0 \\
H^2 & 0 & 0 & 0 & 0 & 0 & 0 & 0 & 0 & 0 & 1 \\
H^3 & 0 & 0 & 1 & 0 & 0 & 0 & 1 & 0 & 0 & 2 \\
H^4 & 0 & 0 & 2 & 0 & 2 & 1 & 2 & 1 & 0 & 3 \\
H^5 & 1 & 2 & 2 & 1 & 2 & 4 & 3 & 3 & 3 & 4 \\
H^6 & 5 & 1 & 1 & 4 & 0 & 3 & 2 & 2 & 5 & 3 \\
\hline
\, & \phi_{210b} & \phi_{216a} & \phi_{280a} & \phi_{280b} & \phi_{315a} & \phi_{336a} & \phi_{378a} & \phi_{405a} & \phi_{420a} & \phi_{512a} \\
\hline
H^0 & 0 & 0 & 0 & 0 & 0 & 0 & 0 & 0 & 0 & 0 \\
H^1 & 0 & 0 & 0 & 0 & 0 & 0 & 0 & 0 & 0 & 0 \\
H^2 & 0 & 0 & 0 & 1 & 0 & 0 & 0 & 0 & 0 & 0 \\
H^3 & 1 & 0 & 0 & 0 & 0 & 0 & 1 & 2 & 2 & 1 \\
H^4 & 4 & 0 & 3 & 1 & 3 & 2 & 3 & 6 & 5 & 4 \\
H^5 & 4 & 4 & 4 & 6 & 5 & 6 & 6 & 6 & 8 & 9 \\
H^6 & 1 & 6 & 3 & 6 & 1 & 6 & 4 & 2 & 6 & 6
\end{array}$}
\end{equation*}
\caption{The cohomology of $\Qt$ as a representation of $\symp{6,\ftwo}$.}
\label{Qtable}

\begin{equation*}
\resizebox{0.8\textwidth}{!}{$
\begin{array}{r|rrrrrrrrrr} 
\, & \phi_{1a} & \phi_{7a} & \phi_{15a} & \phi_{21a} & \phi_{21b} & \phi_{27a} & \phi_{35a} & \phi_{35b} & \phi_{56a} & \phi_{70a} \\
\hline
H^0 & 1&0&0&0&0&0&0&1&0&0\\ 
H^1 & 0&0&0&0&0&1&0&1&0&0\\ 
H^2 & 0&0&0&1&0&0&0&0&0&0\\ 
H^3 & 0&0&0&1&0&0&0&0&0&1\\ 
H^4 & 0&0&0&0&0&1&1&1&0&1\\ 
H^5 & 0&0&1&0&1&1&1&2&0&0 \\
\hline
\, & \phi_{84a} & \phi_{105a} & \phi_{105b} & \phi_{105c} & \phi_{120a} & \phi_{168a} & \phi_{189a} & \phi_{189b} & \phi_{189c} & \phi_{210a} \\
\hline
H^0 & 0&0&0&0&0&0&0&0&0&0\\ 
H^1 & 0&0&0&0&0&1&0&0&0&1\\
H^2 & 0&0&1&0&2&1&1&0&0&3\\
H^3 & 0&0&3&1&3&2&4&1&0&5\\
H^4 & 2&2&3&2&3&6&5&4&4&6\\
H^5 & 4&2&1&4&1&4&3&3&6&4 \\
\hline
\, & \phi_{210b} & \phi_{216a} & \phi_{280a} & \phi_{280b} & \phi_{315a} & \phi_{336a} & \phi_{378a} & \phi_{405a} & \phi_{420a} & \phi_{512a} \\
\hline
H^0 & 0&0&0&0&0&0&0&0&0&0\\
H^1 & 0&0&0&1&0&0&0&0&0&0\\
H^2 & 1&0&0&2&0&0&1&3&2&2\\
H^3 & 5&1&3&3&4&3&5&10&7&7\\
H^4 & 6&5&7&8&7&9&9&10&12&14\\
H^5 & 2&7&4&8&3&8&6&4&8&9
\end{array}$}
\end{equation*}
\caption{The cohomology of $\Htwo{3}$ as a representation of $\symp{6,\ftwo}$.}
\label{Hyptable}
\end{table}

 \clearpage
\bibliographystyle{amsalpha}

\renewcommand{\bibname}{References} 

\bibliography{references} 
\end{document}